\documentclass[12pt]{article}
\usepackage{mathrsfs}
\usepackage[hypertex]{hyperref}
\usepackage{amsfonts}
\usepackage{graphicx}
\usepackage{mathptmx}
\usepackage{latexsym,amsmath,amssymb,amsfonts,amsthm}

\newtheorem{theorem}{Theorem}[section]
\newtheorem{lemma}[theorem]{Lemma}

\newtheorem{proposition}[theorem]{Proposition}
\newtheorem{remark}[theorem]{Remark}
\newtheorem{example}[theorem]{Example}

\newtheorem{defn}[theorem]{Definition}

\begin{document}
\setcounter{page}{1}
\title{Forced hyperbolic mean curvature flow}
\author{Jing Mao}
\date{}
\protect\footnotetext{\!\!\!\!\!\!\!\!\!\!\!\!{ MSC 2010: 58J45;
58J47}
\\
{ ~~Key Words: Hyperbolic mean curvature flow; Hyperbolic partial
differential equation;
Hyperbolic Monge-Amp$\grave{\rm{e}}$re equation; Short-time existence} \\
  Supported by Funda\c{c}\~{a}o para a Ci\^{e}ncia e
Tecnologia (FCT) through a doctoral fellowship SFRH/BD/60313/2009.}
\maketitle ~~~\\[-15mm]
\begin{center}{\footnotesize  Departamento de Matem\'{a}tica,
Instituto Superior T\'{e}cnico, Technical University of Lisbon,
Edif\'{\i}cio Ci\^{e}ncia, Piso 3, Av.\ Rovisco Pais, 1049-001
Lisboa, Portugal; jiner120@163.com, jiner120@tom.com }
\end{center}

\begin{abstract} In this paper, we investigate two hyperbolic flows obtained by
adding forcing terms in direction of the position vector to the
hyperbolic mean curvature flows in \cite{klw,hdl}. For the first
hyperbolic flow, as in \cite{klw}, by using support function, we
reduce it to a hyperbolic Monge-Amp$\grave{\rm{e}}$re equation
successfully, leading to the short-time existence of the flow by the
standard theory of hyperbolic partial differential equation. If the
initial velocity is non-negative and the coefficient function of the
forcing term is non-positive, we also show that there exists a class
of initial velocities such that the solution of the flow exists only
on a finite time interval $[0,T_{max})$, and the solution converges
to a point or shocks and other propagating discontinuities are
generated when $t\rightarrow{T_{max}}$. These generalize the
corresponding results in \cite{klw}. For the second hyperbolic flow,
as in \cite{hdl}, we can prove the system of partial differential
equations related to the flow is strictly hyperbolic, which leads to
the short-time existence of the smooth solution of the flow, and
also the uniqueness. We also derive nonlinear wave equations
satisfied by some intrinsic geometric quantities of the evolving
hypersurface under this hyperbolic flow. These generalize the
corresponding results in \cite{hdl}.
\end{abstract}

\markright{\sl\hfill  J. Mao \hfill}

\section{Introduction}
\renewcommand{\thesection}{\arabic{section}}
\renewcommand{\theequation}{\thesection.\arabic{equation}}
\setcounter{equation}{0} \setcounter{maintheorem}{0}

Generally, we refer to a hyperbolic flow whose main driving factor
is mean curvature as the hyperbolic mean curvature flow (HMCF). In
\cite{hsa}, Rostein, Brandon and Novick-Cohen studied a hyperbolic
mean curvature flow of interfaces and gave a crystalline algorithm
for the motion of closed convex polygonal curves. In \cite{sty}, Yau
has suggested hyperbolic mean curvature flow can be used to model a
vibrating membrane or the motion of a surface. It seems necessary to
study the hyperbolic mean curvature flow because of these
applications.

To our knowledge, few versions of hyperbolic mean curvature flow
have been studied and also few results of these hyperbolic mean
curvature flows have been obtained, see \cite{klw,hdl,pk} for
instance. Now we want to show the motivation why we consider the
hyperbolic mean curvature flows (\ref{1}) and (\ref{2}) below in
this paper. Actually, it is inspired by the similar situation in the
mean curvature flow. More precisely, Ecker and Huisken \cite{kg}
considered the problem that a hypersuface $M_{0}$ immersed in
$R^{n+1}$ evolves by a family of smooth immersions
$X(\cdot,t):M_{0}\rightarrow{R^{n+1}}$ as follows
\begin{eqnarray} \label{1.1}
\left\{
\begin{array}{ll}
\frac{\partial}{\partial{t}}X(x,t)=H(x,t)\vec{N}(x,t),
\quad  \forall{x}\in{M_{0}},~\forall{t}>0\\
X(\cdot,0)=M_{0},  &\quad
\end{array}
\right.
\end{eqnarray}
where $H(x,t)$ and $\vec{N}(x,t)$ are the mean curvature and unit
inner normal vector of the hypersurface
$M_{t}=X(M_{0},t)=X_{t}(M_{0})$, respectively.  If additionally the
initial hypersurface $M_{0}$ is a locally Lipschitz continuous
entire graph over a hyperplane in $R^{n+1}$, they have proved that
the classical mean curvature flow (\ref{1.1}) exists for all the
time $t\in[0,\infty)$, moreover, each $X(\cdot,t)$ is also an entire
graph. Fortunately, by using a similar way, Mao, Li and Wu
\cite{mlw} proved that if the above initial hypersurface, a locally
Lipschitz continuous entire graph in $R^{n+1}$, evolves along the
following curvature flow
\begin{eqnarray} \label{1.2}
\left\{
\begin{array}{ll}
\frac{\partial}{\partial{t}}X(x,t)=H(x,t)\vec{N}(x,t)+\widetilde{c}(t)X(x,t),
\quad  \forall{x}\in{M_{0}},~\forall{t}>0\\
X(\cdot,0)=M_{0},  &\quad
\end{array}
\right.
\end{eqnarray}
where $\widetilde{c}(t)$ is a bounded nonnegative continuous
function, and $H(x,t)$ and $\vec{N}(x,t)$ have the same meanings as
in the flow (\ref{1.1}), then the curvature flow (\ref{1.2}) has
long time existence solutions, and each each $X(\cdot,t)$ is also an
entire graph. This generalizes part of results of Ecker and Huisken,
since if $\widetilde{c}(t)=0$ in (\ref{1.2}), then this flow
degenerates into the classical mean curvature flow (\ref{1.1}).
Similarly, if $\widetilde{c}(t)$ is a bounded continuous function,
for a strictly convex compact hyersurface in $R^{n+1}$ evolving
along the curvature flow of the form (\ref{1.2}), Li, Mao and Wu
\cite{lmw} proved a similar conclusion as in \cite{gs} by mainly
using the methods shown in \cite{gs} and \cite{g}.

Since we could get these nice results if we add a forcing term in
direction of the position vector to the classical mean curvature
flow, we guess maybe it would also work if we add this kind of
forcing term to the hyperbolic mean curvature flows introduced in
\cite{klw} and \cite{hdl} respectively. This process of adding the
forcing term lets us consider the following two initial value
problems.

First, we consider a family of closed plane curves
$F:S^{1}\times[0,T)\rightarrow{R^2}$ which satisfies the following
evolution equation
\begin{eqnarray} \label{1}
\left\{
\begin{array}{lll}
\frac{\partial^{2}F}{\partial{t}^2}(u,t)=k(u,t)\vec{N}(u,t)-\nabla{\rho}+c(t)F(u,t), \quad \forall(u,t)\in{S^{1}\times[0,T)}\\
 F(u,0)=F_{0}(u), \\
 \frac{\partial{F}}{\partial{t}}(u,0)=f(u)\vec{N}_{0},  &\quad
\end{array}
\right.
\end{eqnarray}
where $k(u,t)$ and $\vec{N}(u,t)$ are the curvature and unit inner
normal vector of the plane curve $F(u,t)$ respectively,
$f(u)\in{C^{\infty}(S^{1})}$ is the initial normal velocity, and
$\vec{N}_{0}$ is the unit inner normal vector of the smooth strictly
convex plane curve $F_{0}(u)$. Besides, $c(t)$ is a bounded
continuous function on the interval $[0,T)$ and $\nabla{\rho}$ is
given by
\begin{eqnarray*}
\nabla{\rho}:=\left[\left(\frac{\partial^{2}F}{\partial{s}\partial{t}},\frac{\partial{F}}{\partial{t}}\right)+c(t)(F,\vec{T})\right]\vec{T}(u,t),
\end{eqnarray*}
where $(\cdot,\cdot)$ denotes the standard Euclidean metric in
$R^{2}$, and $\vec{T}$, $s$ denote the unit tangent vector of the
plane curve $F(u,t)$ and the arc-length parameter, respectively.

Fortunately, we can prove the following main results for this flow.

\begin{theorem} \label{theorem1} (Local existence and uniqueness)
For the hyperbolic flow (\ref{1}), there exists a positive constant
$T_{1}>0$ and a family of strictly closed curves $F(\cdot,t)$ with
$t\in[0,T_{1})$ such that each $F(\cdot,t)$ is its solution.
 \end{theorem}

\begin{theorem} \label{theorem2}
For the hyperbolic flow (\ref{1}), if additionally $c(t)$ is
non-positive and the initial velocity $f(u)$ is non-negative, there
exists a class of the initial velocities such that its solution
exists only on a finite time interval $[0,T_{max})$. Moreover, when
$t\rightarrow{T_{max}}$, one of the following must be true
 $\\$(I) the solution $F(\cdot,t)$ converges to a single point, or equivalently, the curvature of the limit curve
 becomes unbounded;
 $\\$(II) the curvature $k(\cdot,t)$ of the curve $F(\cdot,t)$ is
 discontinuous so that the solution converges
 to a piecewise smooth curve, which implies shocks and propagating
 discontinuities may be generated within the hyperbolic flow (\ref{1}).
 \end{theorem}

Second, we consider that an $n$-dimensional smooth  manifold
$\mathscr{M}$ evolves by a family of smooth hypersurface immersions
$X(\cdot,t):\mathscr{M}\rightarrow{R^{n+1}}$ in $R^{n+1}$ as follows
\begin{eqnarray} \label{2}
\left\{
\begin{array}{lll}
\frac{\partial^{2}}{\partial{t^{2}}}X(x,t)=H(x,t)\vec{N}(x,t)+c_{1}(t)X(x,t),
\quad \forall{x}\in{\mathscr{M}}, ~\forall{t}>0\\
X(x,0)=X_{0}(x), \\
 \frac{\partial{X}}{\partial{t}}(x,0)=X_{1}(x),  &\quad
\end{array}
\right.
\end{eqnarray}
where $\vec{N}(x,t)$ is the unit inner normal vector of the
hypersurface $\mathscr{M}_{t}=X(\mathscr{M},t)=X_{t}(\mathscr{M})$,
$X_{0}$ is a smooth hypersurface immersion of $\mathscr{M}$ into
$R^{n+1}$, $X_{1}(x)$ is a smooth vector-valued function on
$\mathscr{M}$, and $c_{1}(t)$ is a bounded continuous function.

For this flow, we can prove the following result.

\begin{theorem} \label{theorem3} (Local existence and uniqueness)
For the hyperbolic flow (\ref{2}), if additionally $\mathscr{M}$ is
compact, then there exists a positive constant $T_{2}>0$ such that
the initial value problem (\ref{2}) has a unique smooth solution
$X(x,t)$ on $\mathscr{M}\times[0,T_{2})$.
 \end{theorem}

 The paper is organized as follows. In Section 2, the notion of support function of $F(u,t)$
 will be introduced, which is used to derive a hyperbolic Monge-Amp$\grave{\rm{e}}$re
 equation leading to the local existence and uniqueness of the
 hyperbolic flow (\ref{1}). An example and some properties of
 the evolving curve have been studied in Section 3. Theorem
 \ref{theorem2} will be proved in Section 4.
 In Section 5, by using the standard existence
 theory of hyperbolic system of partial differential equations, we
 show the short-time existence Theorem \ref{theorem3} of the
 hyperbolic flow (\ref{2}). Some exact solutions of the hyperbolic flow
 (\ref{2}) will be studied in Section 6. The nonlinear wave equations of some
 geometric quantities of the hypersurface $X(\cdot,t)$ will be
 derived in Section 7.

 \section{Proof of theorem \ref{theorem1}}
\renewcommand{\thesection}{\arabic{section}}
\renewcommand{\theequation}{\thesection.\arabic{equation}}
\setcounter{equation}{0} \setcounter{maintheorem}{0}

In this section, we will reparametrize the evolving curves so that
the hyperbolic Monge-Amp$\grave{\rm{e}}$re
 equation could be derived for the support function defined below.
 Reparametrizations can be done since for an evolving curve
 $F(\cdot,t)$ under the flow (\ref{1}), the underlying physics
 should be independent of the choice of the parameter $u\in{S^{1}}$. However, before
 deriving the hyperbolic Monge-Amp$\grave{\rm{e}}$re
 equation, the following definition in \cite{pk} is necessary.

\begin{defn} \label{def1}A flow $F:S^{1}\times[0,T)\rightarrow{R^{2}}$ evolves
normally if and only if its tangential velocity vanishes.
\end{defn}

We claim that our hyperbolic flow (\ref{1}) is a normal flow, since
\begin{eqnarray*}
\frac{d}{dt}\left(\frac{\partial{F}}{\partial{t}},\frac{\partial{F}}{\partial{s}}\right)=-\left(\nabla\rho,\frac{\partial{F}}{\partial{s}}\right)
+c(t)\left(F,\vec{T}\right)+\left(\frac{\partial{F}}{\partial{t}},\frac{\partial^{2}{F}}{\partial{t}\partial{s}}\right)=0,
\end{eqnarray*}
and the initial velocity of the flow (\ref{1}) is in the normal
direction. Then we have
\begin{eqnarray} \label{2.1}
\frac{d}{dt}F(u,t)=\left(\frac{d}{dt}F(u,t),
\vec{N}(u,t)\right)\vec{N}(u,t):=\sigma(u,t)\vec{N}(u,t).
\end{eqnarray}
 By (\ref{1}) and
(\ref{2.1}), we have
\begin{eqnarray} \label{2.3}
\frac{\partial{\sigma}}{\partial{t}}=k(u,t)+c(t)(F,\vec{N})(u,t),
\quad \quad
\sigma\frac{\partial\sigma}{\partial{s}}=\left(\frac{\partial^{2}F}{\partial{s}\partial{t}},\frac{\partial{F}}{\partial{t}}\right),
\end{eqnarray}
where $s=s(\cdot,t)$ is the arc-length parameter of the curve
$F(\cdot,t):S^{1}\rightarrow{R^{2}}$. Obviously, by arc-length
formula, we have
\begin{eqnarray}  \label{2.4}
\frac{\partial}{\partial{s}}=\frac{1}{\sqrt{\left(\frac{\partial{x}}{\partial{u}}\right)^{2}+\left(\frac{\partial{y}}{\partial{u}}\right)^{2}}}
\frac{\partial}{\partial{u}}=\frac{1}{\left|\frac{\partial{F}}{\partial{u}}\right|}\frac{\partial}{\partial{u}}:=\frac{1}{v}\frac{\partial}{\partial{u}},
\end{eqnarray}
here $(x,y)$ is the cartesian coordinate of $R^{2}$. For the
orthogonal frame filed $\{\vec{N},\vec{T}\}$ of $R^{2}$, by Frenet
formula, we have
\begin{eqnarray} \label{2.5}
\frac{\partial{\vec{T}}}{\partial{s}}=k\vec{N}, \quad\quad
\frac{\partial{\vec{N}}}{\partial{s}}=-k\vec{T}.
\end{eqnarray}
Now, in order to give the notion of support function, we have to use
the unit out normal angel, denoted by $\theta$, of a closed convex
curve $F:S^{1}\times[0,T)\rightarrow{R^{2}}$ w.r.t the cartesian
coordinate of $R^{2}$. Then
 \begin{eqnarray*}
 \vec{N}=(-cos\theta,-sin\theta), \quad \quad
 \vec{T}=(-sin\theta,cos\theta),
 \end{eqnarray*}
correspondingly, we have $\frac{\partial\theta}{\partial{s}}=k$ and
\begin{eqnarray} \label{2.13}
\frac{\partial{\vec{N}}}{\partial{t}}=-\frac{\partial{\theta}}{\partial{t}}\vec{T},
\quad
\frac{\partial{\vec{T}}}{\partial{t}}=\frac{\partial{\theta}}{\partial{t}}\vec{N}.
\end{eqnarray}

\begin{lemma} \label{lemma5} The derivative of $v$ with respect to $t$ is
$\frac{\partial{v}}{\partial{t}}=-k\sigma{v}$.
\end{lemma}

\begin{proof}
By using (\ref{2.1}), (\ref{2.4}), and (\ref{2.5}), as in \cite{rm},
we calculate directly as follows
 \begin{eqnarray*}
&&\frac{\partial}{\partial{t}}(v^{2})=2\left(\frac{\partial{F}}{\partial{u}},\frac{\partial^{2}{F}}{\partial{t}\partial{u}}\right)=
2\left(\frac{\partial{F}}{\partial{u}},\frac{\partial^{2}{F}}{\partial{u}\partial{t}}\right)=2\left(v\vec{T},\frac{\partial}{\partial{u}}\left(\sigma\vec{N}\right)\right)=
2\left(v\vec{T},\frac{\partial\sigma}{\partial{u}}\vec{N}-k\sigma{v}\vec{T}\right)\\
&&\qquad \quad =-2v^{2}k\sigma,
 \end{eqnarray*}
 which implies our lemma.
\end{proof}

Then, by using Lemma \ref{lemma5}, we can obtain
\begin{eqnarray*}
\frac{\partial^{2}}{\partial{t}\partial{s}}=\frac{\partial}{\partial{t}}\left(\frac{1}{v}\frac{\partial}{\partial{u}}\right)=
k\sigma\frac{1}{v}\frac{\partial}{\partial{u}}+\frac{1}{v}\frac{\partial}{\partial{u}}\frac{\partial}{\partial{t}}=
k\sigma\frac{\partial}{\partial{s}}+\frac{\partial^{2}}{\partial{s}\partial{t}},
 \end{eqnarray*}
 which implies
 \begin{eqnarray*}
 \frac{\partial{\vec{T}}}{\partial{t}}=\frac{\partial}{\partial{t}}\left(\frac{\partial{F}}{\partial{s}}\right)\vec{N}=
 \frac{\partial\sigma}{\partial{s}}\vec{N}.
 \end{eqnarray*}
Combining this equality with (\ref{2.13}) yields
$\frac{\partial{\theta}}{\partial{t}}=\frac{\partial\sigma}{\partial{s}}$.

Assume $F(u,t):S^{1}\times{[0,T)}\rightarrow{R^{2}}$ is a family of
convex curves satisfying the flow (\ref{1}). Now, as in \cite{x}, we
will use the normal angel to reparametrize the evolving curve
$F(\cdot,t)$, and then give the notion of support function which is
used to derive the local existence of the flow (\ref{1}). Set
\begin{eqnarray} \label{2.6}
\widetilde{F}(\theta,\tau)=F(u(\theta,\tau),t(\theta,\tau)),
\end{eqnarray}
where $t(\theta,\tau)=\tau$. We claim that under the parametrization
(\ref{2.6}), $\vec{N}$ and $\vec{T}$ are independent of the
parameter $\tau$.  In fact, by chain rule we have
 \begin{eqnarray*}
 0=\frac{\partial\theta}{\partial\tau}=\frac{\partial\theta}{\partial{u}}\frac{\partial{u}}{\partial\tau}
 +\frac{\partial\theta}{\partial{t}},
 \end{eqnarray*}
which implies
\begin{eqnarray*}
\frac{\partial\theta}{\partial{t}}=-\frac{\partial\theta}{\partial{u}}\frac{\partial{u}}{\partial\tau}=-
\frac{\partial\theta}{\partial{s}}\frac{\partial{s}}{\partial{u}}\frac{\partial{u}}{\partial\tau}=-kv\frac{\partial{u}}{\partial\tau}.
\end{eqnarray*}
Therefore,
\begin{eqnarray*}
\frac{\partial\vec{T}}{\partial\tau}=\frac{\partial\vec{T}}{\partial{t}}+
\frac{\partial\vec{T}}{\partial{s}}\frac{\partial{s}}{\partial{u}}\frac{\partial{u}}{\partial\tau}=\left(\frac{\partial\theta}{\partial{t}}
+kv\frac{\partial{u}}{\partial\tau}\right)\vec{N}=0.
\end{eqnarray*}
Similarly, we have
$\frac{\partial\vec{N}}{\partial\tau}=-\left(\frac{\partial\theta}{\partial{t}}
+kv\frac{\partial{u}}{\partial\tau}\right)\vec{T}=0$, then our claim
follows.

Define the support function of the evolving curve
$\widetilde{F}(\theta,\tau)=\left(x(\theta,\tau),y(\theta,\tau)\right)$
as follows
\begin{eqnarray*}
S(\theta,\tau)=\left(\widetilde{F}(\theta,\tau),-\vec{N}\right)=x(\theta,\tau)cos\theta+y(\theta,\tau)sin\theta,
\end{eqnarray*}
consequently,
\begin{eqnarray*}
S_{\theta}(\theta,\tau)=-x(\theta,\tau)sin\theta+y(\theta,\tau)cos\theta=\left(\widetilde{F}(\theta,\tau),\vec{T}\right).
\end{eqnarray*}
Therefore, we have
\begin{eqnarray} \label{2.7}
\left\{
\begin{array}{ll}
x(\theta,\tau)=Scos\theta-S_{\theta}sin\theta,\\
y(\theta,\tau)=Ssin\theta+S_{\theta}cos\theta,  &\quad
\end{array}
\right.
\end{eqnarray}
which implies the curve $\widetilde{F}(\theta,\tau)$ can be
represented by the support function. Then we have
\begin{eqnarray} \label{2.8}
S_{\theta\theta}+S=-x_{\theta}sin\theta+y_{\theta}cos\theta=\left(\frac{\partial\widetilde{F}}{\partial\theta},\vec{T}\right)=
\left(\frac{\partial\widetilde{F}}{\partial{s}}\frac{\partial{s}}{\partial\theta},\vec{T}\right)=\frac{1}{k},
\end{eqnarray}
since the evolving curve
$\widetilde{F}(\theta,\tau)=F(u(\theta,\tau),t(\theta,\tau))$ is
strictly convex, (\ref{2.8}) makes sense.

On the other hand, since $\vec{N}$ and $\vec{T}$ are independent of
the parameter $\tau$, together with (\ref{2.1}) and (\ref{2.6}), we
have
\begin{eqnarray}
S_{\tau}=\left(\frac{\partial\widetilde{F}}{\partial\tau},-\vec{N}\right)=\left(\frac{\partial{F}}{\partial{u}}\frac{\partial{u}}{\partial\tau}
+\frac{\partial{F}}{\partial{t}},\vec{N}\right)=\left(\frac{\partial{F}}{\partial{t}},-\vec{N}\right)=-\sigma(u,t),
\end{eqnarray}  \label{2.9}
furthermore, by chain rule we obtain
 \begin{eqnarray*}
S_{\tau\tau}&=&\left(\frac{\partial{F}}{\partial{u}}\frac{\partial^{2}{u}}{\partial\tau^{2}}+
\frac{\partial^{2}F}{\partial{u}^{2}}\left(\frac{\partial{u}}{\partial\tau}\right)^{2}+2\frac{\partial^{2}F}{\partial{u}\partial\tau}
\frac{\partial{u}}{\partial\tau}+\frac{\partial^{2}F}{\partial{t}^{2}},-\vec{N}\right)\\
&=&\left(\frac{\partial^{2}F}{\partial{u}^{2}}\left(\frac{\partial{u}}{\partial\tau}\right)^{2}+\frac{\partial^{2}F}{\partial{u}\partial\tau}
\frac{\partial{u}}{\partial\tau},-\vec{N}\right)+\left(\frac{\partial^{2}F}{\partial{u}\partial\tau}
\frac{\partial{u}}{\partial\tau}+\frac{\partial^{2}F}{\partial{t}^{2}},-\vec{N}\right)\\
&=&\left(\left(\frac{\partial{F}}{\partial{u}}\right)_{\tau},-\vec{N}\right)\frac{\partial{u}}{\partial\tau}+
\left(\frac{\partial^{2}F}{\partial{u}\partial\tau}
\frac{\partial{u}}{\partial\tau},-\vec{N}\right)-k-c(\tau)\left(F,\vec{N}\right)\\
&=&\left(\frac{\partial^{2}F}{\partial{u}\partial\tau}
\frac{\partial{u}}{\partial\tau},-\vec{N}\right)-k+c(\tau)S(\theta,\tau).
\end{eqnarray*}
Since $F(u,t):S^{1}\times[0,T)\rightarrow{R^{2}}$ is a normal flow,
which implies
\begin{eqnarray*}
\left(\frac{\partial{F}}{\partial{t}},\vec{T}\right)(u,t)\equiv0,
\end{eqnarray*}
for all $t\in[0,T)$. By straightforward computation,  we have
\begin{eqnarray*}
S_{\theta\tau}=\left(\frac{\partial^{2}F}{\partial{u}\partial{t}}\frac{\partial{u}}{\partial\theta},-\vec{N}\right)
=\frac{1}{kv}\left(\frac{\partial^{2}F}{\partial{u}\partial{t}},-\vec{N}\right),
\end{eqnarray*}
and
\begin{eqnarray*}
S_{\tau\theta}=\left(\frac{\partial\widetilde{F}}{\partial\tau},\vec{T}\right)=\left(\frac{\partial{F}}{\partial{u}}
\frac{\partial{u}}{\partial\tau}+\frac{\partial{F}}{\partial{t}},\vec{T}\right)=v\frac{\partial{u}}{\partial\tau}.
\end{eqnarray*}
Hence, the support function $S(\theta,\tau)$ satisfies
\begin{eqnarray*}
S_{\tau\tau}=\left(\frac{\partial^{2}F}{\partial{u}\partial\tau}
\frac{\partial{u}}{\partial\tau},-\vec{N}\right)-k+c(\tau)S(\theta,\tau)=kv\frac{\partial{u}}{\partial\tau}S_{\theta\tau}
-k+c(\tau)S(\theta,\tau)=k(S_{\theta\tau}^{2}-1)+c(\tau)S,
\end{eqnarray*}
combining this equality with (\ref{2.8}) yields
\begin{eqnarray} \label{2.10}
S_{\tau\tau}=\frac{S_{\theta\tau}^{2}-1}{S_{\theta\theta}+S}+c(\tau)S,
\quad\quad  \forall(\theta,\tau)\in{S^{1}}\times[0,T).
\end{eqnarray}
Then it follows from (\ref{1}), (\ref{2.6}), (\ref{2.10}) that
\begin{eqnarray} \label{2.11}
\left\{
\begin{array}{lll}
SS_{\tau\tau}-c(\tau)SS_{\theta\theta}+(S_{\tau\tau}S_{\theta\theta}-S_{\theta\tau}^{2})+1-c(\tau)S^{2}=0,\\
S(\theta,0)=(F_{0},-\vec{N})=h(\theta),  \\
S_{\tau}(\theta,0)=-\widetilde{f}(\theta)=-f(u(\theta,0)), &\quad
\end{array}
\right.
\end{eqnarray}
where $h(\theta)$ and $\widetilde{f}(\theta)$ are the support
functions of the initial curve $F_{0}(u(\theta))$ and the initial
velocity of this initial curve, respectively.

Now, we want to use the conclusion of the hyperbolic
Monge-Amp$\grave{\rm{e}}$re
 equation to get the short-time existence of the flow (\ref{1}).
 Actually, for an unknown function $z(\theta,\tau)$ with two variables $\theta$, $\tau$,
 its Monge-Amp$\grave{\rm{e}}$re
 equation has the form
 \begin{eqnarray} \label{2.12}
 A+Bz_{\tau\tau}+Cz_{\tau\theta}+Dz_{\theta\theta}+E\left(z_{\tau\tau}z_{\theta\theta}-z_{\theta\tau}^{2}\right)=0,
 \end{eqnarray}
here the coefficients $A,B,C,D,E$ depend on
$\tau,\theta,z,z_{\tau},z_{\theta}$. (\ref{2.12}) is said to be
$\tau$-hyperbolic for $S$, if
$\triangle^{2}(\tau,\theta,z,z_{\tau},z_{\theta}):=C^{2}-4BD+4AE>0$
and $z_{\theta\theta}+B(\tau,\theta,z,z_{\tau},z_{\theta})\neq0$. We
also need to require the $\tau$-hyperbolicity at the initial time,
in fact, if we rewrite the initial values as
$z(\theta,0)=z_{0}(\theta)$, $z_{\tau}(\theta,0)=z_{1}(\theta)$ for
the unknown function $z(\theta,\tau)$, $\theta\in[0,2\pi]$, then the
corresponding $\tau$-hyperbolic condition is given as follows
\begin{eqnarray*}
\triangle^{2}(0,\theta,z_{0},z_{1},z'_{0})=(C^{2}-4BD+4AE)|_{t=0}>0,
\end{eqnarray*}
\begin{eqnarray*}
z''_{0}+B(0,\theta,z_{0},z_{1},z'_{0})\neq0 ,
\end{eqnarray*}
where $z'_{0}=\frac{dz_{0}}{d\theta}$,
$z''_{0}=\frac{d^{2}z_{0}}{d\theta^{2}}$.

It is easy to check that (\ref{2.11}) is a hyperbolic
Monge-Amp$\grave{\rm{e}}$re equation. In fact, for (\ref{2.11}),
\begin{eqnarray*}
A=1-c(\tau)S^{2}, \quad B=S, \quad C=0, \quad D=-c(\tau)S, \quad
E=1,
\end{eqnarray*}
then we have
\begin{eqnarray*}
\triangle^{2}(\tau,\theta,S,S_{\tau},S_{\theta})=C^{2}-4BD+4AE=0^{2}-4S\times(-c(\tau)S)+4(1-c(\tau)S^{2})\times1=4>0,
\end{eqnarray*}
and
\begin{eqnarray*}
S_{\theta\theta}+B(\tau,\theta,S,S_{\tau},S_{\theta})=S_{\theta\theta}+S=\frac{1}{k}\neq0.
\end{eqnarray*}
Furthermore, if at least $h(\theta)\in{C^{3}([0,2\pi])}$ and
$\widetilde{f}(\theta)\in{C^{2}([0,2\pi])}$, then we have
\begin{eqnarray*}
\triangle^{2}(0,\theta,h,\widetilde{f},h_{\theta})=4>0,
\end{eqnarray*}
and
\begin{eqnarray*}
h_{\theta\theta}+B(0,\theta,h,\widetilde{f},h_{\theta})\neq0,
\end{eqnarray*}
which implies (\ref{2.11}) is also $\tau$-hyperbolic at $\tau=0$.
Hence, (\ref{2.11}) is a hyperbolic Monge-Amp$\grave{\rm{e}}$re
equation.

Then by the standard theory of hyperbolic equations (e.g.,
\cite{l,m}), Theorem \ref{theorem1} concerning the local existence
and uniqueness of the solution of the hyperbolic flow (\ref{1})
follows.

 \section{Some properties of the flow (\ref{1})}
\renewcommand{\thesection}{\arabic{section}}
\renewcommand{\theequation}{\thesection.\arabic{equation}}
\setcounter{equation}{0} \setcounter{maintheorem}{0}

First, we would like to give an example so that we could understand
the hyperbolic flow (\ref{1}) deeply, however, first we need the
following lemma

\begin{lemma} \label{lemma1}
Consider the initial value problem
\begin{eqnarray} \label{3.1}
\left\{
\begin{array}{ll}
r_{tt}=-\frac{c_{0}}{r}+\bar{c}(t)r\\
r(0)=r_{0}>0, \quad r_{t}(0)=r_{1},  &\quad
\end{array}
\right.
\end{eqnarray}
where $c_{0}$ is a positive constant and $\bar{c}(t)$ is a
non-positive bounded continuous function. For arbitrary initial data
$r_{0}>0$, if the initial velocity $r_{1}\leq0$, then the solution
$r=r(t)$ decreases and attains its zero point at time $t_{0}$ (in
particular, when $r_{1}=0$, we have
$t_{0}\leq\sqrt{\frac{\pi}{2c_{0}}}r_{0}$, equality holds iff
$\bar{c}(t)=0$); if the initial velocity is positive, then the
solution $r$ increases first and then decreases and attains its zero
point at a finite time.
\end{lemma}

\begin{proof}
The proof is similar with the arguments in \cite{hdl,hsa}. The
discussion is divided into two cases.

$\\$ Case (I). The initial velocity is non-positive, i.e.
$r_{1}\leq0$.

Assume $r(t)>0$ for all the time $t>0$. Then by (\ref{3.1}) we have
$r_{tt}=-\frac{c_{0}}{r}+\bar{c}(t)r<0$, then by monotonicity
$r_{t}(t)<r_{t}(0)=r_{1}\leq0$ for all $t>0$. Hence, there exists a
time $t_{0}$ such that $r(t_{0})=0$, which is contradict with our
assumption. Moreover, when the initial velocity vanishes, i.e.
$r_{t}(0)=r_{1}=0$, let $c^{+}$ be the bound of the function
$\bar{c}(t)$, i.e. $|\bar{c}(t)|\leq{c^{+}}$ for all $t>0$,
obviously, multiplying both sides of
$r_{tt}=-\frac{c_{0}}{r}+\bar{c}(t)r$ by $r_{t}$, integrating from 0
to $t<t_{0}$, applying the
 conditions $r_{t}(0)=r_{1}=0$ yields
 \begin{eqnarray} \label{3.2}
c_{0}\ln\frac{r_{0}}{r}\leq\frac{r_{t}^{2}}{2}\leq{c_{0}\ln\frac{r_{0}}{r}}+\frac{c^{+}}{2}(r_{0}^{2}-r(t)^{2}),
 \end{eqnarray}
integrating both sides of (\ref{3.2}) on the interval $[0,t_{0}]$
and using the condition $r(t_{0})=0$ yields
\begin{eqnarray*}
\frac{\sqrt{\pi}}{2}=\int^{\infty}_{0}e^{-u^{2}}du\geq\int^{t_{0}}_{0}\frac{\sqrt{2c_{0}}}{2r_{0}}dt\geq\int^{\infty}_{0}e^{-u^{2}}du-\frac{\sqrt{c^{+}}}{2}\int^{t_{0}}_{0}
\sqrt{\frac{r_{0}^{2}-r(t)^{2}}{\ln\frac{r_{0}}{r}}}r_{0}^{-1}dt,
 \end{eqnarray*}
 where $u=\sqrt{\ln\frac{r_{0}}{r}}$. Therefore, we obtain
 \begin{eqnarray} \label{3.3}
 \sqrt{\frac{\pi}{2c_{0}}}r_{0}-\frac{Ar_{0}}{\sqrt{2c_{0}}}\leq{t_{0}}\leq
 \sqrt{\frac{\pi}{2c_{0}}}r_{0},
 \end{eqnarray}
where \begin{eqnarray*} A=\sqrt{c^{+}}\int^{t_{0}}_{0}
\sqrt{\frac{r_{0}^{2}-r(t)^{2}}{\ln\frac{r_{0}}{r}}}r_{0}^{-1}dt.
\end{eqnarray*}
Obviously, equalities in (\ref{3.3}) hold simultaneously if and only
if $c^{+}=0$, which implies $\bar{c}(t)=0$, in this case,
$t_{0}=\sqrt{\frac{\pi}{2c_{0}}}r_{0}$, which is a conclusion in
\cite{hdl} for $c_{0}=1$.

$\\$ Case (II). The initial velocity is positive, i.e. $r_{1}>0$.
From (\ref{3.1}), we have
\begin{eqnarray} \label{3.4}
r_{t}^{2}(t)=-2c_{0}\left(\ln{r(t)}-\ln{r_{0}}\right)+r_{1}^{2}+2\int^{t}_{0}\bar{c}(s)rr_{t}(s)ds.
\end{eqnarray}
Assume $r$ increases all the time, i.e. $r_{t}>0$ for all the time
$t>0$. Since $r\geq{r_{0}}>0$, $r_{t}>0$ and $\bar{c}(t)$ is
non-positive, then from (\ref{3.4}) we obtain
\begin{eqnarray*}
r_{t}^{2}(t)\leq-2c_{0}\ln\frac{r}{r_{0}}+r_{1}^{2},
\end{eqnarray*}
which implies
\begin{eqnarray} \label{3.5}
r_{0}\leq{r(t)}\leq{e^{\frac{r_{1}^{2}}{2c_{0}}}}r_{0}.
\end{eqnarray}
On the other hand, under our assumption, we have
\begin{eqnarray*}
-\frac{c_{0}}{r}-c^{+}r\leq{r_{tt}}\leq-\frac{c_{0}}{r},
\end{eqnarray*}
combining this relation with (\ref{3.5}) results in
 \begin{eqnarray*}
B(r_{0})\leq{r_{tt}}\leq-e^{-\frac{r_{1}^{2}}{2c_{0}}}\frac{c_{0}}{r_{0}},
 \end{eqnarray*}
 where
 \begin{eqnarray*}
 B(r_{0})=\min\left\{-\frac{c_{0}}{r_{0}}-c^{+}r_{0}, -e^{-\frac{r_{1}^{2}}{2c_{0}}}\frac{c_{0}}{r_{0}}-
 c^{+}e^{\frac{r_{1}^{2}}{2c_{0}}}r_{0}\right\}<0.
 \end{eqnarray*}
 Thus the curve $r_{t}$ can be bounded by two straight lines
 $r_{t}=B(r_{0})t+r_{1}$ and
 $r_{t}=-e^{-\frac{r_{1}^{2}}{2c_{0}}}\frac{c_{0}}{r_{0}}t+r_{1}$,
 which implies $r_{t}$ must be negative for
 $t>\frac{r_{1}r_{0}e^{\frac{r_{1}^{2}}{2c_{0}}}}{c_{0}}$. This is
 contradict with our assumption. Hence, $r_{t}$ will change sign and
 becomes negative at certain finite time, which implies there exist a finite
 time $t_{1}$ such that $r_{t}(t_{1})=0$. Now, if we assume  $r(t)>0$ for all the time $t>0$,
  then as in the case (I), we can prove $r(t)$ attains its zero point at a finite time
 $t_{2}>t_{1}$. Thus in this case $r(t)$ increases first and then
 decreases and attains its zero point at a finite time. Our conclusion follows by the above arguments.
\end{proof}

\begin{example}  \label{example1}  \rm{ Suppose $c(t)$ in the hyperbolic
mean curvature flow (\ref{1}) is also non-positive, and $F(\cdot,t)$
in (\ref{1}) is a family of round circles with radius $r(t)$
centered at the origin. More precisely,
\begin{eqnarray*}
F(u,t)=r(t)(cosu,sinu), \quad r(0)>0,
\end{eqnarray*}
without loss of generality, we can also choose $u=s$ to be the
arc-length parameter of the curve $F(\cdot,t)$. Then the curvature
$k(\cdot,t)$ of the evolving curve $F(\cdot,t)$ is $\frac{1}{r(t)}$,
moreover, $\nabla{\rho}=0$. Substituting these into (\ref{1}) yields
\begin{eqnarray} \label{3.6}
\left\{
\begin{array}{ll}
r_{tt}=-\frac{1}{r}+c(t)r,\\
r(0)=r_{0}>0, \quad   r_{t}=r_{1}.     &\quad
\end{array}
\right.
\end{eqnarray}
By Lemma \ref{lemma1}, we know if the initial velocity $r_{1}\leq0$,
then the flow (\ref{1}) shrinks and converges to a single point at a
finite time $t_{0}$ (in particular, when $r_{1}=0$,
$t_{0}\leq\sqrt{\frac{\pi}{2}}r_{0}$, equality holds iff $c(t)=0$);
if the initial velocity is positive, then the flow (\ref{1}) expands
first and shrinks and converges to a single point at a finite time.
One can also interpret this phenomenon by physical principle as in
\cite{klw,hdl}. }
\end{example}
\begin{remark} \rm{From this example, we know the necessity of the
non-positivity of the bounded continuous function $c(t)$ if we want
to get the convergence of the hyperbolic flow (\ref{1}). That is the
motivation why we add the condition $c(t)$ is non-positive in the
Theorem \ref{theorem2} to try to get the convergence.}
\end{remark}

Inspired by Chou's basic idea \cite{ks} for proving the convergence
of the curve shortening flow, by using the maximum principle of the
second order hyperbolic partial differential equations shown in
\cite{pw}, we could get the following conclusions as proposition 3.1
and proposition of preserving convexity in \cite{klw}. This is true,
since, comparing with the evolution equations in the proofs of
proposition 3.1 and proposition of preserving convexity in
\cite{klw}, one can easily check that the corresponding evolution
equations of the difference of the support functions and the
curvature function under the flow (\ref{1}) only have extra first
order terms $c(t)w$ and $-c(t)k$ respectively, moreover, these first
order terms have no affection on the usage of the maximum principle.

\begin{proposition} \label{proposition1}(Containment principle)
Suppose $F_{1}$ and $F_{2}:S^{1}\times[0,T_{1})\rightarrow{R^{2}}$
are convex solutions of (\ref{1}). If $F_{2}(\cdot,0)$ lies in the
domain enclosed by $F_{1}(\cdot,0)$ and $f_{2}(u)\geq{f_{1}(u)}$,
then $F_{2}(\cdot,t)$ is contained in the domain enclosed by
$F_{1}(\cdot,t)$ for all $t\in[0,T_{1})$.
\end{proposition}

\begin{proposition} (Preserving convexity) \label{proposition2}
Let $k_{0}$ be the mean curvature of the initial curve $F_{0}$, and
let $\eta=\min\limits_{\theta\in[0,2\pi]}k_{0}(\theta)$. Then, for a
$C^{4}$-solution of (\ref{2.11}), one has
\begin{eqnarray} \label{3.16}
k(\theta,t)\geq\eta:=\min\limits_{\theta\in[0.2\pi]}k_{0}(\theta),
\quad for~~ t\in[0,T_{max}), \quad \theta\in[0,2\pi],
\end{eqnarray}
 where $k(\theta,t)$ is the mean curvature
of the evolving curve $F(\cdot,t)$, and $[0,T_{max})$ is the maximal
time interval of the solution $F(\cdot,t)$ of (\ref{1}).
\end{proposition}

 \section{Convergence}
\renewcommand{\thesection}{\arabic{section}}
\renewcommand{\theequation}{\thesection.\arabic{equation}}
\setcounter{equation}{0} \setcounter{maintheorem}{0}

In this section, we want to get the convergence of the hyperbolic
flow (\ref{1}). We assume $c(t)$ is non-positive and initial
velocity $f(u)$ is non-negative. In order to get the convergence,
the following lemma is needed.
 \begin{lemma} \label{lemma3} The arclength $\mathfrak{L}(t)$ of the
 evolving closed curve $F(\cdot,t)$ of the flow (\ref{1}) satisfies
 \begin{eqnarray*}
 \frac{d\mathfrak{L}(t)}{dt}=-\int_{0}^{2\pi}\widetilde{\sigma}(\theta,t)d\theta,
 \end{eqnarray*}
 and
 \begin{eqnarray*}
 \frac{d^{2}\mathfrak{L}(t)}{dt^{2}}=\int_{0}^{2\pi}\left[\left(\frac{\partial\widetilde{\sigma}}{\partial\theta}\right)^{2}k-k+c(t)S\right]d\theta,
 \end{eqnarray*}
 where
 $\widetilde{\sigma}(\theta,t)=\widetilde{\sigma}(\theta,\tau)=\sigma(u,t)$,
  the change of variables from $(u,t)$ to $(\theta,\tau)$ satisfies (\ref{2.6}).
 \end{lemma}

 \begin{proof} The
 convention of using $t$ for time variable is used here. In addition, by straightforward computation, we have
 \begin{eqnarray*}
\frac{d\mathfrak{L}(t)}{dt}=\frac{d}{dt}\int_{S^{1}}v(u,t)d{u}=\int_{S^{1}}\frac{d}{dt}v(u,t)d{u}
=-\int_{S^{1}}k\sigma{v}du=-\int_{0}^{2\pi}\widetilde{\sigma}d\theta,
 \end{eqnarray*}
 and
 \begin{eqnarray*}
\frac{d^{2}\mathfrak{L}(t)}{dt^{2}}&=&-\int_{0}^{2\pi}\frac{\partial}{\partial{t}}\left(\widetilde{\sigma}(\theta,t)\right)d\theta=\int_{0}^{2\pi}\left[(S_{\theta{t}}^{2}-1)k+c(t)S\right]d\theta\\
&=&\int_{0}^{2\pi}\left[\left(\left(\frac{\partial\widetilde\sigma}{\partial{\theta}}\right)^{2}-1\right)k+c(t)S\right]d\theta,
 \end{eqnarray*}
 here $v(u,t)$ is defined in (\ref{2.4}), $u\in{S^{1}}$, and the fact $\frac{\partial}{\partial{t}}v(u,t)=
 -k{\sigma}v$ is shown in Lemma \ref{lemma5}. Therefore, our proof is completed.
 \end{proof}

\emph{Proof of theorem 1.2}. Let $[0,T_{max})$ be the maximal time
interval for the solution $F(\cdot,t)$ of the flow (\ref{1}) with
$F_{0}$ and $f$ as initial curve and the initial velocity,
respectively. We divide the proof into five steps.

\vskip 2mm
 {Step 1. Preserving convexity}
\vskip 1mm
 By Proposition \ref{proposition2}, we know the evolving curve
 $F(S^{1},t)$ remains strictly convex
 and the curvature of $F(S^{1},t)$ has a uniformly positive lower
 bound $\min\limits_{\theta\in[0,2\pi]}k_{0}(\theta)$ on
 $S^{1}\times[0,T_{max})$.

\vskip 2mm {Step 2. Short-time existence} \vskip 1mm
 Without loss of generality, we can assume the origin $o$ of $R^{2}$ is
 in the exterior of the domain enclosed by the initial curve
 $F_{0}$. Enclose the initial curve $F_{0}$ by a large enough round circle
 $\gamma_{0}$ centered at $o$, and then let this circle evolve under the flow
 (\ref{1}) with the initial velocity $\min\limits_{u\in{S^{1}}}f(u)$ to get a solution
 $\gamma(\cdot,t)$. From the Example \ref{example1}, we know the
 solution $\gamma(\cdot,t)$ exists only at a finite time interval
 $[0,T_{0})$, and $\gamma(\cdot,t)$ shrinks into a point as
 $t\rightarrow{T_{0}}$. By Proposition \ref{proposition1}, we know
 that $F(\cdot,t)$ is always enclosed by $\gamma(\cdot,t)$ for all
 $t\in[0,T_{0})$. Therefore, we have that the solution
 $F(\cdot,t)$ must become singular at some time $T_{max}\leq{T_{0}}$.

\vskip 2mm
 {Step 3. Hausdorff convergence}
 \vskip 1mm
 As in \cite{klw,ks,x}, we also want to use a classical result,
 Blaschke Selection Theorem, in convex geometry (c.f. \cite{rs}).
$\\$(Blaschke Selection Theorem) \emph{Let $\{K_{j}\}$ be a sequence
of convex sets which are contained in a bounded set. Then there
exists a subsequence $\{K_{jk}\}$ and a convex set $K$ such that
$K_{jk}$ converges to $K$ in the Hausdorff metric.}
 \vskip 1mm

 The round circle $\gamma_{0}$ in the step 2 is shrinking under the flow
 (\ref{1}), since the normal initial velocity $f$ is non-negative,
 this conclusion can be easily obtained from Lemma \ref{lemma1}.
 Since for every time $t\in[0,T_{max})$, $F(\cdot,t)$ is enclosed by
 $\gamma(\cdot,t)$, we have every convex set $K_{F(\cdot,t)}$ enclosed by
 $F(\cdot,t)$ is contained in a bounded set $K_{\gamma_{0}}$
 enclosed by $\gamma_{0}$. Thus, by Blaschke Selection Theorem, we can
 directly conclude that $F(\cdot,t)$ converges to a (maybe degenerate and
 nonsmooth) weakly convex curve $F(\cdot,T_{max})$ in the Hausdorff
 metric.

\vskip 2mm
 {Step 4. Length of evolving curve}
 \vskip 1mm
We claim that there exists a finite time $\bar{T}\leq\infty$ such
that $\mathfrak{L}(\bar{T})=0$.

As the step 2, we can easily find a round circle $\bar{\gamma_{0}}$
center at the origin $o$ enclosed by the convex initial curve
$F_{0}$, and then let this circle evolve under the flow
 (\ref{1}) with the initial velocity $\max\limits_{u\in{S^{1}}}f(u)$ to get a solution
 $\bar{\gamma}(\cdot,t)$. From the Example \ref{example1}, we know the
 solution $\bar{\gamma}(\cdot,t)$ exists only at a finite time interval
 $[0,\bar{T}_{0})$ with $\bar{T}_{0}\leq{T_{max}}$, and $\bar{\gamma}(\cdot,t)$ shrinks into a point as
 $t\rightarrow{\bar{T_{0}}}$. By Proposition \ref{proposition1}, we know
 that $F(\cdot,t)$ always encloses $\bar{\gamma}(\cdot,t)$ for all
 $t\in[0,\bar{T}_{0})$. Thus we know that the support function
 $S(\theta,t)$ is nonnegative on the time interval
 $[0,\bar{T}_{0})$, and we can also conclude that
 $\widetilde{\sigma}(\theta,t)=\sigma(u,t)$ is also nonnegative on the interval
 $[0,\bar{T}_{0})$, since
\begin{eqnarray} \label{4.1}
\frac{\partial{\sigma}}{\partial{t}}=k(u,t)+c(t)(F,\vec{N})(u,t)>0
\end{eqnarray}
and $\sigma(u,0)=f(u)\geq0$. The expression (\ref{4.1}) holds since
$k$ has a uniformly positive lower bound, $c(t)$ is non-positive,
and $(F,\vec{N})=-S\leq0$ on the time interval $[0,\bar{T}_{0})$.
Hence, we have
 \begin{eqnarray} \label{4.2}
\frac{d\mathfrak{L}(t)}{dt}=-\int_{0}^{2\pi}\widetilde{\sigma}d\theta<0,
 \end{eqnarray}
on the time interval $[0,\bar{T}_{0})$.

On the other hand, since $\sigma(u,t)>\sigma(u,0)$ for all
$t\in(0,\bar{T}_{0})$, which implies
\begin{eqnarray*}
\widetilde{\sigma}(\theta,t)=\sigma(u,t)>\widetilde{\sigma}(\theta,0)=
\sigma(u,0), \quad for~ all ~t\in(0,\bar{T}_{0}),
\end{eqnarray*}
so we have
\begin{eqnarray} \label{4.3}
\frac{\partial\widetilde{\sigma}}{\partial{t}}(\theta,t)>0,
\end{eqnarray}
for all $t\in(0,\bar{T}_{0})$. Combining (\ref{4.3}) with the truth
\begin{eqnarray*}
\frac{\partial\sigma}{\partial{t}}(u,t)=\frac{\partial\widetilde{\sigma}}{\partial\theta}(\theta,t)\cdot\frac{\partial\theta}{\partial{t}}
+\frac{\partial\widetilde{\sigma}}{\partial{t}}(\theta,t)=\frac{\partial\widetilde{\sigma}}{\partial\theta}\cdot\frac{\partial\sigma}{\partial{s}}
+\frac{\partial\widetilde{\sigma}}{\partial{t}}(\theta,t)=
\left(\frac{\partial\widetilde{\sigma}}{\partial\theta}\right)^{2}(\theta,t)\cdot\frac{\partial\theta}{\partial{s}}+\frac{\partial\widetilde{\sigma}}{\partial{t}}(\theta,t)
\end{eqnarray*}
yields
\begin{eqnarray*}
\frac{\partial\widetilde{\sigma}}{\partial{t}}=k\left[1-\left(\frac{\partial\widetilde{\sigma}}{\partial\theta}\right)^{2}\right]-c(t)S>0,
\end{eqnarray*}
which indicates
\begin{eqnarray} \label{4.4}
\frac{d^{2}\mathfrak{L}(t)}{dt^{2}}=\int_{0}^{2\pi}
\left[\left(\left(\frac{\partial\widetilde\sigma}{\partial{t}}\right)^{2}-1\right)k+c(t)S\right]d\theta<0
\end{eqnarray}
on the time interval $(0,\bar{T}_{0})$.

 Then our claim follows from the facts $\mathfrak{L}(0)>0$, (\ref{4.2}) and (\ref{4.4}).

\vskip 2mm
 {Step 5. Convergence}
 \vskip 1mm

This step is the same as the step 4 of the proof of theorem 4.1 in
\cite{klw}. Our proof is finished.              $~\square$

 \section{Short time existence of the flow (\ref{2})}
\renewcommand{\thesection}{\arabic{section}}
\renewcommand{\theequation}{\thesection.\arabic{equation}}
\setcounter{equation}{0} \setcounter{maintheorem}{0}

In this section, we would like to give the short time existence of
the solution of the hyperbolic mean curvature flow (\ref{2}) by
using the method shown in \cite {hdl}.

Now, consider the hyperbolic flow (\ref{2}), additionally, we assume
$\mathscr{M}$ is a compact Riemannian manifold. Endow the
n-dimensional smooth compact manifold $\mathscr{M}$ with a local
coordinate system $\{x^{i}\}$, $1\leq{i}\leq{n}$. Denote by
$\{g_{ij}\}$ and $\{h_{ij}\}$ the induced metric and the second
fundamental form on $\mathscr{M}$ respectively, then the mean
curvature is given by
\begin{eqnarray*}
H=g^{ij}h_{ij},
\end{eqnarray*}
where $(g^{ij})$ is the inverse of the metric matrix $(g_{ij})$.

As the mean curvature flow (MCF) case, here we want to use a trick
of DeTurck \cite{dd} to show that the evolution equation
\begin{eqnarray} \label{5.1}
\frac{\partial^{2}}{\partial{t^{2}}}X(x,t)=H(x,t)\vec{N}(x,t)+c_{1}(t)X(x,t),
\end{eqnarray}
in (\ref{2}) is strictly hyperbolic, then we can use the standard
existence theory of the hyperbolic equations to get the short-time
existence of our flow (\ref{2}). However, first we would like to
rewrite (\ref{5.1}) in terms of the coordinate components.

Denote by $\nabla$ and $\triangle$ the Riemannian connection and
Beltrami-Laplacian operator on $\mathscr{M}$ decided by the induced
metric $\{g_{ij}\}$, respectively. Let $(\cdot,\cdot)$ be the
standard Euclidean metric of $R^{n+1}$. Recall that in this case the
Gauss-Weingarten relations of submanifold can be rewritten as
follows
\begin{eqnarray} \label{5.6}
\frac{\partial^{2}X}{\partial{x^{i}}\partial{x^{j}}}=\Gamma_{ij}^{k}\frac{\partial{X}}{\partial{x^{k}}}+h_{ij}\vec{n},
\quad
\frac{\partial\vec{n}}{\partial{x^{j}}}=-h_{jl}g^{lm}\frac{\partial{X}}{\partial{x^{m}}},
\end{eqnarray}
where $\vec{n}$ is the unit inward normal vector field on
$\mathscr{M}$, and $\Gamma_{ij}^{k}$ is the Christoffel symbol of
the Riemannian connection $\nabla$, moreover,
$\Gamma_{ij}^{k}=g^{kl}\left(\frac{\partial^{2}X}{\partial{x^{i}}\partial{x^{j}}},\frac{\partial{X}}{\partial{x^{l}}}\right)$.
Therefore, we have
\begin{eqnarray*}
\triangle{X}=g^{ij}\nabla_{i}\nabla_{j}X=g^{ij}\left(\frac{\partial^{2}X}{\partial{x^{i}}\partial{x^{j}}}-\Gamma_{ij}^{k}\frac{\partial{X}}{\partial{x^{k}}}\right)
=g^{ij}h_{ij}\vec{n}=H\vec{n},
\end{eqnarray*}
which implies the evolution equation (\ref{5.1}) can be equivalently
rewritten as
\begin{eqnarray} \label{5.2}
\frac{\partial^{2}X}{\partial{t}^{2}}=g^{ij}\frac{\partial^{2}X}{\partial{x^{i}}\partial{x^{j}}}-g^{ij}g^{kl}
\left(\frac{\partial^{2}X}{\partial{x^{i}}\partial{x^{j}}},\frac{\partial{X}}{\partial{x^{l}}}\right)\frac{\partial{X}}{\partial{x^{k}}}
+c_{1}(t)X.
\end{eqnarray}
However, it is easy to see (\ref{5.2}) is not strictly hyperbolic,
since the Laplacian is taken in the induced metric which changes
with $X(\cdot,t)$, and this adds extra terms to the symbol. One
could get the detailed explanation in Chapter 2 of \cite{x}.

Now, we need to use the trick of DeTurck, modifying the flow
(\ref{2}) through a diffeomorphism of $\mathscr{M}$, to construct a
strictly hyperbolic equation, leading to the short-time existence.
Suppose $\bar{X}(x,t)$ is a solution of equation (\ref{5.1}) (or
equivalently (\ref{5.2})) and
$\phi_{t}:\mathscr{M}\rightarrow\mathscr{M}$ is a family of
diffeomorphisms of $\mathscr{M}$. Let
 \begin{eqnarray} \label{5.3}
 X(x,t)=\phi_{t}^{\ast}\bar{X}(x,t),
 \end{eqnarray}
where $\phi_{t}^{\ast}$ is the pull-back operator of $\phi_{t}$, and
denote the diffeomorphism $\phi_{t}$ by
\begin{eqnarray*}
(y,t)=\phi_{t}(x,t)=\left\{y^{1}(x,t),y^{2}(x,t),\ldots,y^{n}(x,t)\right\}
\end{eqnarray*}
in the local coordinates. In what follows, we need to show the
existence of the the diffeomorphism $\phi_{t}$, and the equations
satisfied by $X(x,t)$ is strictly hyperbolic, which leads to the
short-time existence of $X(x,t)$, together with the existence of
$\phi_{t}$ and (\ref{5.3}), we could obtain the short-time existence
of $\bar{X}(x,t)$, which is assumed to be the solution of the flow
(\ref{2}). That is to say through this process we can get the
short-time existence of the flow (\ref{2}).

As in \cite{hdl}, consider the following initial value problem
\begin{eqnarray} \label{5.4}
\left\{
\begin{array}{lll}
\frac{\partial^{2}y^{\alpha}}{\partial{t}^{2}}=\frac{\partial{y}^{\alpha}}{\partial{x^{k}}}
\left(g^{ij}(\Gamma_{ij}^{k}-\tilde{\Gamma}_{ij}^{k})\right),\\
\\
 y^{\alpha}(x,0)=x^{\alpha}, \quad y_{t}^{\alpha}(x,0)=0,
&\quad
\end{array}
\right.
\end{eqnarray}
where $\tilde{\Gamma}_{ij}^{k}$ is the Christoffel symbol related to
the initial metric
$\tilde{g}_{ij}=\left(\frac{\partial{X}}{\partial{x^{i}}},\frac{\partial{X}}{\partial{x^{j}}}\right)(x,0)$.
Since
 \begin{eqnarray} \label{5.5}
\Gamma_{ij}^{k}=\frac{\partial{y^{\alpha}}}{\partial{x^{j}}}\frac{\partial{y^{\beta}}}{\partial{x^{l}}}
\frac{\partial{x^{k}}}{\partial{y^{\gamma}}}\bar{\Gamma}_{\alpha\beta}^{\gamma}+\frac{\partial{x^{k}}}{\partial{y^{\alpha}}}
\frac{\partial^{2}y^{\alpha}}{\partial{x^{j}}\partial{x^{l}}},
 \end{eqnarray}
which implies the initial problem (\ref{5.4}) can be rewritten as
\begin{eqnarray*}
\left\{
\begin{array}{lll}
\frac{\partial^{2}y^{\alpha}}{\partial{t}^{2}}=g^{ij}\left(\frac{\partial^{2}y^{\alpha}}{\partial{x^{j}}\partial{x^{l}}}
+\frac{\partial{y^{\beta}}}{\partial{x^{j}}}\frac{\partial{y^{\gamma}}}{\partial{x^{l}}}
\bar{\Gamma}_{\alpha\beta}^{\gamma}-\frac{\partial{y^{\alpha}}}{\partial{x^{j}}}\tilde{\Gamma}_{ij}^{k}\right),\\
\\
y^{\alpha}(x,0)=x^{\alpha}, \quad y_{t}^{\alpha}(x,0)=0, &\quad
\end{array}
\right.
\end{eqnarray*}
which is an initial value problem for a strictly hyperbolic system.
By the standard existence theory of a hyperbolic system, we know
there must exist a family of diffeomorphisms $\phi_{t}$ which
satisfies the initial value problem (\ref{5.4}).

On the other hand, by (\ref{5.5}), we have
\begin{eqnarray*}
\triangle_{\bar{g}}\bar{X}&=&\bar{g}^{\alpha\beta}\nabla_{\alpha}\nabla_{\beta}\bar{X}\\
&=&g^{kl}\frac{\partial^{2}X}{\partial{x^{k}}\partial{x^{l}}}+g^{kl}\frac{\partial{y^{\alpha}}}{\partial{x^{k}}}
\frac{\partial{y^{\beta}}}{\partial{x^{l}}}\frac{\partial{X}}{\partial{x^{i}}}\frac{\partial^{2}x^{i}}
{\partial{y^{\alpha}}\partial{y^{\beta}}}-g^{kl}\frac{\partial{X}}{\partial{x^{i}}}\left(\Gamma^{i}_{kl}-\frac{\partial{x^{i}}}
{\partial{y^{\gamma}}}\frac{\partial^{2}y^{\gamma}}{\partial{x^{k}}\partial{x^{l}}}\right)\\
&=&g^{kl}\nabla_{k}\nabla_{l}X=\triangle_{g}X,
\end{eqnarray*}
and then
\begin{eqnarray*}
\frac{\partial^{2}X}{\partial{t}^{2}}&=&\frac{\partial^{2}\bar{X}}{\partial{y^{\alpha}}\partial{y^{\alpha}}}
\frac{\partial{y^{\alpha}}}{\partial{t}}\frac{\partial{y^{\beta}}}{\partial{t}}+2\frac{\partial^{2}\bar{X}}{\partial{t}\partial{y^{\beta}}}
\frac{\partial{y^{\beta}}}{\partial{t}}+\frac{\partial^{2}\bar{X}}{\partial{t}^{2}}+\frac{\partial{\bar{X}}}{\partial{y^{\alpha}}}
\frac{\partial^{2}y^{\alpha}}{\partial{t^{2}}}\\
&=&\triangle_{g}X+c_{1}(t)\bar{X}+\frac{\partial{X}}{\partial{y^{\alpha}}}g^{ij}\left(\Gamma_{ij}^{k}-\tilde{\Gamma}_{ij}^{k}\right)+\frac{\partial^{2}\bar{X}}{\partial{y^{\alpha}}\partial{y^{\alpha}}}
\frac{\partial{y^{\alpha}}}{\partial{t}}\frac{\partial{y^{\beta}}}{\partial{t}}+2\frac{\partial^{2}\bar{X}}{\partial{t}\partial{y^{\beta}}}
\frac{\partial{y^{\beta}}}{\partial{t}} \\
&=&g^{ij}\frac{\partial^{2}X}{\partial{x^{i}}\partial{x^{j}}}-g^{ij}\tilde{\Gamma}_{ij}^{k}\frac{\partial{X}}{\partial{x^{k}}}+\frac{\partial^{2}\bar{X}}{\partial{y^{\alpha}}\partial{y^{\alpha}}}
\frac{\partial{y^{\alpha}}}{\partial{t}}\frac{\partial{y^{\beta}}}{\partial{t}}+2\frac{\partial^{2}\bar{X}}{\partial{t}\partial{y^{\beta}}}
\frac{\partial{y^{\beta}}}{\partial{t}}+c_{1}(t)\bar{X},
\end{eqnarray*}
which is strictly hyperbolic. Hence, by the standard existence
theory of hyperbolic equations (see \cite{l}), we could get the
short-time existence of $X(x,t)$, then by what we have point out
before this directly leads to the short-time existence of the
solution, $\bar{X}(x,t)$, of the equation (\ref{5.1}), which implies
our local existence and uniqueness Theorem \ref{theorem3} naturally.

\section{Examples}
\renewcommand{\thesection}{\arabic{section}}
\renewcommand{\theequation}{\thesection.\arabic{equation}}
\setcounter{equation}{0} \setcounter{maintheorem}{0}

In this section, by using Lemma \ref{lemma1}, we investigate the
exact solution of examples given in \cite{hdl}, and find that we
could get the similar results, which implies our hyperbolic flow
(\ref{2}) is meaningful.

\begin{example} \rm{Suppose $c_{1}(t)$ in the hyperbolic flow (\ref{2}) is non-positive. Now, consider a family of spheres
\begin{eqnarray*}
X(x,t)=r(t)(cos\alpha{cos\beta},cos\alpha{sin\beta},sin\alpha),
\end{eqnarray*}
where $\alpha\in[-\frac{\pi}{2},\frac{\pi}{2}]$, $\beta\in[0,2\pi]$.
By straightforward computation, we have the induced metric and the
second fundamental form are
\begin{eqnarray*}
g_{11}=r^{2}, \quad g_{22}=r^{2}cos^{2}\alpha, \quad
g_{12}=g_{21}=0,
\end{eqnarray*}
and
\begin{eqnarray*}
h_{11}=r, \quad h_{22}=rcos^{2}\alpha, \quad h_{12}=h_{21}=0,
\end{eqnarray*}
respectively. So, the mean curvature is
\begin{eqnarray*}
H=g^{ij}h_{ij}=\frac{2}{r}.
\end{eqnarray*}
Additionally, the unit inward normal vector of each $F(\cdot,t)$ is
$\vec{n}=-(cos\alpha{cos\beta},cos\alpha{sin\beta},sin\alpha)$,
hence our hyperbolic flow (\ref{2}) becomes
\begin{eqnarray*}
\left\{
\begin{array}{ll}
r_{tt}=-\frac{2}{r}+c_{1}(t)r\\
r(0)=r_{0}>0, \quad r_{t}(0)=r_{1},  &\quad
\end{array}
\right.
\end{eqnarray*}
then by Lemma \ref{lemma1}, we know for arbitrary $r(0)=r_{0}>0$, if
the initial velocity $r_{t}(0)=r_{1}>0$, the evolving sphere will
expand first and then shrink to a single point at a finite time; if
the initial velocity $r_{t}(0)=r_{1}\leq0$, the evolving sphere will
shrink to a point directly at a finite time. One could also use the
physical principle to interpret this phenomenon as in \cite{hdl},
which is very simple.}
\end{example}

\begin{example} \rm{Suppose $c_{1}(t)$ in the hyperbolic flow (\ref{2}) is
 non-positive. Now, consider a family of round circles
\begin{eqnarray*}
X(x,t)=(r(t)cos\alpha,r(t)sin\alpha),
\end{eqnarray*}
where $\alpha\in[0,2\pi]$. It is easy to find that the mean
curvature and the unit inward normal vector of each $X(\cdot,t)$ are
$\frac{1}{r(t)}$ and $\vec{n}=-(cos\alpha,sin\alpha)$, respectively,
then our hyperbolic flow (\ref{2}) becomes
\begin{eqnarray*}
\left\{
\begin{array}{ll}
r_{tt}=-\frac{1}{r}+c_{1}(t)r\\
r(0)=r_{0}>0, \quad r_{t}(0)=r_{1},  &\quad
\end{array}
\right.
\end{eqnarray*}
then by Lemma \ref{lemma1}}, we know that the circles will shrink to
a point at a finite time for arbitrary $r(0)>0$ and the initial
velocity $r_{1}$.
\end{example}

\begin{remark} \rm{Comparing with the example 2 in \cite{hdl}, here we would like to point out the hyperbolic flow
(\ref{2}) does not have cylinder solution except $c_{1}(t)\equiv0$.
In fact, suppose the solution of the flow (\ref{2}), $X(\cdot,t)$,
is a family of cylinders which takes form
\begin{eqnarray} \label{6.1}
X(x,t)=(r(t)cos\alpha,r(t)sin\alpha,\rho),
\end{eqnarray}
where $\alpha\in[0,2\pi]$ and $\rho\in[0,\rho_{0}]$, then as before
we could obtain $c_{1}(t)\rho=0$ directly, which implies our claim
here. Why the the hyperbolic flow (\ref{2}) does not have cylinder
solution of the form (\ref{6.1}) if $c_{1}(t)$ dose not vanish? We
think that is because the term $c_{1}(t)X(\cdot,t)$ not only has
component perpendicular to $\rho$-axis, which lets the cylinder move
toward $\rho$-axis vertically, but also has component parallel with
$\rho$-axis, which leads to the moving of cylinder along the
$\rho$-axis. This fact implies, after initial time, the hyperbolic
flow (\ref{2}) will change the shape of the initial cylinder such
that the evolving surface $X(\cdot,t)$ is not cylinder any more. }
\end{remark}

\section{Evolution equations}
\renewcommand{\thesection}{\arabic{section}}
\renewcommand{\theequation}{\thesection.\arabic{equation}}
\setcounter{equation}{0} \setcounter{maintheorem}{0}

In this section, we would like to give the evolution equations for
some intrinsic quantities of the hypersurface $X(\cdot,t)$ under the
hyperbolic mean curvature flow (\ref{2}), which will be important
for the future study, like convergence, on this flow. It is not
difficult to derive them, since they just have slight changes
comparing with corresponding the evolution equations in \cite{hdl}.

First, from \cite{x}, we have the following facts for hypersurface
\begin{lemma} \label{lemma4}Under the hyperbolic mean curvature flow (\ref{2}),
the following identities hold
\begin{eqnarray}
\triangle{h_{ij}}=\nabla_{i}\nabla_{j}H+Hh_{il}g^{lm}h_{mj}-|A|^{2}h_{ij},
\end{eqnarray}
\begin{eqnarray}
\triangle|A|^{2}=2g^{ik}g^{jl}h_{kl}\nabla_{i}\nabla_{j}H+2|\nabla{A}|^{2}+2Htr(A^{3})-2|A|^{4},
\end{eqnarray}
where
\begin{eqnarray*}
|A|^{4}=g^{ij}g^{kl}h_{ik}h_{jl},  \quad
tr(A^{3})=g^{ij}g^{kl}g^{mn}h_{ik}h_{lm}h_{nj}.
\end{eqnarray*}
\end{lemma}

\begin{theorem} \label{theorem4}Under the hyperbolic mean curvature flow (\ref{2}),
we have
 \begin{eqnarray} \label{8.3}
\frac{\partial^{2}g_{ij}}{\partial{t^{2}}}=-2Hh_{ij}+2c_{1}(t)g_{ij}+2\left(\frac{\partial^{2}X}{\partial{t}\partial{x^{i}}},
\frac{\partial^{2}X}{\partial{t}\partial{x^{j}}}\right),
 \end{eqnarray}
 $$\frac{\partial^{2}\vec{n}}{\partial{t^{2}}}=-g^{ij}\frac{\partial{H}}{\partial{x^{i}}}\frac{\partial{X}}{\partial{x^{j}}}+g^{ij}\left(\vec{n},\frac{\partial^{2}X}{\partial{t}\partial{x^{i}}}\right)
 \qquad  \qquad\qquad  \qquad$$
\begin{eqnarray}   \label{8.4} \qquad\qquad \qquad
\times\left[2g^{kl}\left(\frac{\partial{X}}{\partial{x^{j}}},\frac{\partial^{2}X}{\partial{t}\partial{x^{l}}}\right)\frac{\partial{X}}{\partial{x^{k}}}+
g^{kl}\left(\frac{\partial{X}}{\partial{x^{l}}},\frac{\partial^{2}X}{\partial{t}\partial{x^{j}}}\right)\frac{\partial{X}}{\partial{x^{k}}}-\frac{\partial^{2}X}{\partial{t}\partial{x^{j}}}\right],
 \end{eqnarray}
 and
 $$\frac{\partial^{2}h_{ij}}{\partial{t^{2}}}=\triangle{h_{ij}}-2Hh_{il}h_{mj}g^{lm}+|A|^{2}h_{ij}+g^{kl}h_{ij}
\left(\vec{n},\frac{\partial^{2}X}{\partial{t}\partial{x^{k}}}\right)\left(\vec{n},\frac{\partial^{2}X}{\partial{t}\partial{x^{l}}}\right)$$
\begin{eqnarray} \label{8.5}
-2\frac{\partial\Gamma_{ij}^{k}}{\partial{t}}\left(\vec{n},\frac{\partial^{2}X}{\partial{t}\partial{x^{k}}}\right)+c_{1}(t)h_{ij}.
 \end{eqnarray}
\end{theorem}

\begin{proof} By the definition of the induced
metric and (\ref{5.6}), we have
\begin{eqnarray*}
&&\frac{\partial^{2}g_{ij}}{\partial{t^{2}}}=\left(\frac{\partial^{3}X}{\partial{t^{2}}\partial{x^{i}}},\frac{\partial{X}}{\partial{x^{j}}}\right)
+2\left(\frac{\partial^{2}X}{\partial{t}\partial{x^{i}}},\frac{\partial^{2}X}{\partial{t}\partial{x^{j}}}\right)+
\left(\frac{\partial{X}}{\partial{x^{i}}},\frac{\partial^{3}X}{\partial{t^{2}}\partial{x^{j}}}\right)\\
&&\quad~\quad
=\left(\frac{\partial}{\partial{x^{i}}}\left(H\vec{n}+c_{1}(t)X\right),\frac{\partial{X}}{\partial{x^{j}}}\right)+2\left(\frac{\partial^{2}X}{\partial{t}\partial{x^{i}}},\frac{\partial^{2}X}{\partial{t}\partial{x^{j}}}\right)
+\left(\frac{\partial{X}}{\partial{x^{i}}},\frac{\partial}{\partial{x^{j}}}\left(H\vec{n}+c_{1}(t)X\right)\right)\\
&&\quad~\quad=H\left(-h_{ik}g^{kl}\frac{\partial{X}}{\partial{x^{l}}},\frac{\partial{X}}{\partial{x^{j}}}\right)+2c_{1}(t)\left(\frac{\partial{X}}{\partial{x^{i}}},\frac{\partial{X}}{\partial{x^{j}}}\right)
+2\left(\frac{\partial^{2}X}{\partial{t}\partial{x^{i}}},\frac{\partial^{2}X}{\partial{t}\partial{x^{j}}}\right)+\\
&&\quad~\quad\quad H\left(\frac{\partial{X}}{\partial{x^{i}}},-h_{jk}g^{kl}\frac{\partial{X}}{\partial{x^{l}}}\right)\\
&&\quad~\quad
=-2Hh_{ij}+2c_{1}(t)g_{ij}+2\left(\frac{\partial^{2}X}{\partial{t}\partial{x^{i}}},
\frac{\partial^{2}X}{\partial{t}\partial{x^{j}}}\right),
\end{eqnarray*}
which finishes the proof of (\ref{8.3}).

It is surprising that the evolution equation for the unit inward
normal vector $\vec{n}$ under the flow (\ref{2}) here has no
difference with the one in \cite{hdl}, since in the process of
deriving the evolution equation for $\vec{n}$, the only possible
difference appears in the term
\begin{eqnarray*}
-\left(\vec{n},\frac{\partial^{3}X}{\partial{t^{2}}\partial{x^{i}}}\right)g^{ij}\frac{\partial{X}}{\partial{x^{j}}}=
-\left(\vec{n},\frac{\partial}{\partial{x^{i}}}\left(H\vec{n}+c_{1}(t)X\right)\right)g^{ij}\frac{\partial{X}}{\partial{x^{j}}}=-g^{ij}
\frac{\partial{H}}{\partial{x^{i}}}\frac{\partial{X}}{\partial{x^{j}}}.
\end{eqnarray*}
However, this is the same with the case in \cite{hdl}, since the
term
\begin{eqnarray*}
-\left(\vec{n},\frac{\partial}{\partial{x^{i}}}\left(c_{1}(t)X\right)\right)g^{ij}\frac{\partial{X}}{\partial{x^{j}}}
\end{eqnarray*}
vanishes. So, (\ref{8.4}) follows according to the corresponding
evolution equation in \cite{hdl}.

Actually, (\ref{8.5}) is easy to be obtained by comparing with the
proof of evolution equation (5.5) in \cite{hdl}, since, between our
case and the case in \cite{hdl}, one could find that the processes
of deriving the evolution equations only have slight difference.
However, the deriving process in \cite{hdl} is a little complicated,
so we would like to give the detailed steps here so that readers can
note the difference clearly. By (\ref{5.6}), we have
\begin{eqnarray*}
\frac{\partial{h_{ij}}}{\partial{t}}=\frac{\partial}{\partial{t}}\left(\vec{n},\frac{\partial^{2}X}{\partial{x^{i}}
\partial{x^{j}}}\right)=\left(\frac{\partial{\vec{n}}}{\partial{t}},\frac{\partial^{2}X}{\partial{x^{i}}\partial{x^{j}}}\right)+
\left(\vec{n},\frac{\partial^{3}X}{\partial{t}\partial{x^{i}}\partial{x^{j}}}\right),
\end{eqnarray*}
furthermore,
\begin{eqnarray*}
&&\frac{\partial^{2}h_{ij}}{\partial{t}^{2}}=\left(\frac{\partial^{2}{\vec{n}}}{\partial{t}^{2}},\frac{\partial^{2}X}{\partial{x^{i}}\partial{x^{j}}}\right)
+2\left(\frac{\partial{\vec{n}}}{\partial{t}},\frac{\partial^{3}X}{\partial{t}\partial{x^{i}}\partial{x^{j}}}\right)+
\left(\vec{n},\frac{\partial^{4}X}{\partial{t}^{2}\partial{x^{i}}\partial{x^{j}}}\right)\\
&&\quad~\quad=-g^{kl}\left(\frac{\partial{H}}{\partial{x^{k}}}\frac{\partial{X}}{\partial{x^{l}}},\frac{\partial^{2}X}{\partial{x^{i}}\partial{x^{j}}}\right)-
g^{kl}\left(\vec{n},\frac{\partial^{2}X}{\partial{t}\partial{x^{k}}}\right)\left(\frac{\partial^{2}X}{\partial{t}\partial{x^{l}}},\frac{\partial^{2}X}{\partial{x^{i}}\partial{x^{j}}}\right)\\
&&\quad~\quad\quad+g^{pq}g^{kl}\left(\vec{n},\frac{\partial^{2}X}{\partial{t}\partial{x^{p}}}\right)\left[\left(\frac{\partial{X}}{\partial{x^{l}}},\frac{\partial^{2}X}{\partial{t}\partial{x^{q}}}\right)
+2\left(\frac{\partial{X}}{\partial{x^{q}}},\frac{\partial^{2}X}{\partial{t}\partial{x^{l}}}\right)\right]\left(\frac{\partial{X}}{\partial{x^{k}}},\frac{\partial^{2}X}{\partial{x^{i}}\partial{x^{j}}}\right)\\
&&\quad~\quad\quad-2g^{kl}\left(\vec{n},\frac{\partial^{2}X}{\partial{t}\partial{x^{k}}}\right)\left(\frac{\partial{X}}{\partial{x^{l}}},\frac{\partial^{3}X}{\partial{t}\partial{x^{i}}\partial{x^{j}}}\right)+
\left(\vec{n},\frac{\partial}{\partial{x^{i}}\partial{x^{j}}}\left(H\vec{n}+c_{1}(t)X\right)\right),
\end{eqnarray*}
then one could easily find that the difference between our case and
the case in \cite{hdl} appears from the last term
\begin{eqnarray*}
\left(\vec{n},\frac{\partial}{\partial{x^{i}}\partial{x^{j}}}\left(H\vec{n}+c_{1}(t)X\right)\right),
\end{eqnarray*}
which satisfies
\begin{eqnarray*}
\left(\vec{n},\frac{\partial}{\partial{x^{i}}\partial{x^{j}}}\left(H\vec{n}+c_{1}(t)X\right)\right)&=&\left(\vec{n},\frac{\partial}{\partial{x^{i}}}
\left(\frac{\partial{H}}{\partial{x^{j}}}\vec{n}-Hh_{jk}g^{kl}\frac{\partial{X}}{\partial{x^{l}}}+c_{1}(t)\frac{\partial{X}}{\partial{x^{j}}}\right)\right)\\
&=&\left(\vec{n},\frac{\partial}{\partial{x^{i}}}
\left(\frac{\partial{H}}{\partial{x^{j}}}\vec{n}-Hh_{jk}g^{kl}\frac{\partial{X}}{\partial{x^{l}}}\right)\right)+c_{1}(t)\left(\vec{n},\frac{\partial^{2}X}{\partial{x^{i}}
\partial{x^{j}}}\right)\\
&=&\left(\vec{n},\frac{\partial}{\partial{x^{i}}}
\left(\frac{\partial{H}}{\partial{x^{j}}}\vec{n}-Hh_{jk}g^{kl}\frac{\partial{X}}{\partial{x^{l}}}\right)\right)+c_{1}(t)h_{ij}.
\end{eqnarray*}
Obviously, it will only produce an extra term $c_{1}(t)h_{ij}$
comparing with the evolution equation for the second fundamental
form, (5.5), in \cite{hdl}. So, the evolution equation (\ref{8.5})
follows.
\end{proof}

At the end, by Lemma \ref{lemma4} and Theorem \ref{theorem4}, we
could derive the following evolution equations for the mean
curvature and the square norm of the second fundamental form of the
hypersurface $X(\cdot,t)$, which maybe play an important role in the
future study, like convergence, of the hyperbolic flow (\ref{2}) as
the mean curvature flow case.
\begin{theorem} Under the hyperbolic mean curvature flow (\ref{2}),
we have
\begin{eqnarray}
\frac{\partial^{2}H}{\partial{t}^{2}}=\triangle{H}+H|A|^{2}-2g^{ik}g^{jl}
\left(\frac{\partial^{2}X}{\partial{t}\partial{x^{k}}},\frac{\partial^{2}X}{\partial{t}\partial{x^{l}}}\right)+
Hg^{kl}\left(\vec{n},\frac{\partial^{2}X}{\partial{t}\partial{x^{k}}}\right)\left(\vec{n},\frac{\partial^{2}X}{\partial{t}\partial{x^{l}}}\right)\nonumber\\
-2g^{ij}\frac{\partial\Gamma_{ij}^{k}}{\partial{t}}\left(\vec{n},\frac{\partial^{2}X}{\partial{t}\partial{x^{k}}}\right)+
2g^{ik}g^{jp}g^{lq}h_{ij}\frac{\partial_{g_{pq}}}{\partial{t}}\frac{\partial_{g_{kl}}}{\partial{t}}-
2g^{ik}g^{jl}\frac{\partial_{g_{kl}}}{\partial{t}}\frac{\partial_{h_{ij}}}{\partial{t}}-c_{1}(t)H,
\label{7.6}
\end{eqnarray}
 and
 \begin{eqnarray}
\frac{\partial^{2}}{\partial{t}^{2}}|A|^{2}=\triangle(|A|^{2})-2|\nabla{A}|^{2}+2|A|^{4}+2|A|^{2}g^{pq}
\left(\vec{n},\frac{\partial^{2}X}{\partial{t}\partial{x^{p}}}\right)\left(\vec{n},\frac{\partial^{2}X}{\partial{t}\partial{x^{q}}}\right)\nonumber\\
+2g^{ij}g^{kl}\frac{\partial{h_{ik}}}{\partial{t}}\frac{\partial{h_{jl}}}{\partial{t}}-8g^{im}g^{jn}g^{kl}h_{jl}
\frac{g_{mn}}{\partial{t}}\frac{h_{ik}}{\partial{t}}\nonumber\\
-4g^{im}g^{jn}g^{kl}h_{ik}h_{jl}\left(\vec{n},\frac{\partial^{2}X}{\partial{t}\partial{x^{m}}}\right)\left(\vec{n},\frac{\partial^{2}X}{\partial{t}\partial{x^{n}}}\right)
+2g^{im}\frac{\partial{g_{pq}}}{\partial{t}}\frac{\partial{g_{mn}}}{\partial{t}}h_{ik}h_{jl}\nonumber\\
\times\left(2g^{jp}g^{nq}g^{kl}+g^{jn}g^{kp}g^{lq}\right)-4g^{ij}g^{kl}h_{jl}\frac{\partial\Gamma_{ik}^{p}}{\partial{t}}
\left(\vec{n},\frac{\partial^{2}X}{\partial{t}\partial{x^{p}}}\right)-2c_{1}(t)|A|^{2}.
\label{7.7}
 \end{eqnarray}
\end{theorem}

\vskip 1mm
\begin{proof} Here we do not give the detailed proof, since in
\cite{hdl} they have given the detailed and straightforward
computation on how to derive the evolution equations. Moreover, in
our case we find that if we want to get our theorem here, we only
need to use the evolution equations (\ref{8.3}) and (\ref{8.5}) for
the induced metric and the second fundamental form to replace the
old ones in \cite{hdl} in the computation.
\end{proof}

\begin{remark} \rm{Here we want to point out an interesting truth. In \cite{mlw,lmw}, we have proved
$\\$\emph {Lemma (\cite{mlw,lmw}). If the hypersuface $X(\cdot,t)$
of $R^{n+1}$ satisfies the curvature flow of the form (\ref{1.2}),
then
 $\\$ (1)$\frac{\partial}{\partial{t}}g_{ij}=-2Hh_{ij}+2\widetilde{c}(t)g_{ij}$,
 $\\$ (2)$\frac{\partial}{\partial{t}}\vec{v}=\nabla^{i}H\cdot\frac{\partial{X}}{\partial{x^{i}}}$,
 $\\$ (3)$\frac{\partial}{\partial{t}}h_{ij}=\triangle{h_{ij}}-2Hh_{il}g^{lm}h_{mj}+|A|^{2}h_{ij}+\widetilde{c}(t)h_{ij}$,
 $\\$ (4)$\frac{\partial}{\partial{t}}H=\triangle{H}+|A|^{2}H-\widetilde{c}(t)H$,
 $\\$ (5)$\frac{\partial}{\partial{t}}|A|^{2}=\triangle{|A|^{2}}-2|\nabla{A}|^{2}+2|A|^{4}-2\widetilde{c}(t)|A|^{2}$,
 $\\$ where $\vec{v}$ denotes the unit outward normal vector of $X(\cdot,t)$.}
$\\$ Comparing with those corresponding evolution equations derived
by Huisken in \cite{g}, the extra terms  are
$2\widetilde{c}(t)g_{ij}$, $0$, $\widetilde{c}(t)h_{ij}$,
$-\widetilde{c}(t)H$, and $-2\widetilde{c}(t)|A|^{2}$, if we add a
forcing term, $\widetilde{c}(t)X$, to the evolution equation of the
mean curvature flow in direction of the position vector. However,
the surprising truth is that if we add this forcing term to the
hyperbolic flow in \cite{hdl}, we find that no matter how
complicated the evolution equations of the intrinsic quantities of
the hypersurface $X(\cdot,t)$ under the hyperbolic flow (\ref{2})
are, the evolution equations (\ref{8.3})-(\ref{7.7}) also have the
extra terms of the same forms as (1)-(5) comparing with the
corresponding evolution equations in \cite{hdl}.}
\end{remark}


\begin{thebibliography}{99}
\bibitem{klw} D. Kong, K. Liu and Z. Wang, Hyperbolic mean curvature flow: evolution of plane curves,
\emph{Acta Mathematica Scientia}
  {\bf 29B} (3) (2009) 493--514.

\bibitem{hdl} C. He, D. Kong and K. Liu, Hyperbolic mean curvature flow,
\emph{J. Differential Equations} {\bf 246} (2009) 373--390.

\bibitem{mlw} J. Mao, G. Li and C. Wu, Entire graphs
under a general flow, \emph{Demonstratio Mathematica}
  {\bf XLII} (2009) 631--640.

\bibitem{lmw} G. Li, J. Mao and C. Wu, Convex mean curvature flow with a
forcing term in direction of the position vector, \emph{Acta
Mathematica Sinica, English Series} {\bf 28} (2) (2012) 313--332.


\bibitem{hsa} H. G. Rotstein, S. Brandon and A. Novick-Cohen, Hyperbolic flow by mean curvature,
\emph{Journal of Krystal Growth} {\bf 198-199} (1999) 1256--1261.

\bibitem{sty} S.T. Yau, Review of geometry and analysis, \emph{Asian J. Math}. {\bf 4} (2000) 235--278.

\bibitem{pk} P. G. LeFloch, K. Smoczyk, The hyperbolic mean curvature flow, \emph{J. Math. Pures
Appl.} {\bf 90} (2008) 591--614.

\bibitem{kg} K. Ecker, G. Huisken, Mean curvature evolution of entire graphs, \emph{Ann. of
Math.} {\bf 130} (1989) 453--471.

\bibitem{gs} G. Li, I. Salavessa, Forced convex mean curvature flow in Euclidean spaces, \emph{Manuscript
Math.} {\bf 126} (2008) 335--351.

\bibitem{g} G. Huisken, Forced by mean curvature of convex surfaces into spheres,
\emph{J. Differential Geom.} {\bf 20} (1984) 237--266.

\bibitem{ks} K. S. Chou(K. Tso), Deforming a hypersurface by its Gauss-Kroneck curvature,
\emph{Comm. Pure Appl. Math.} {\bf 38} (1985) 867--882.

\bibitem{x} X. Zhu, Lectures on mean curvature flow, AMS/IP Studies
in Advaned Mathematics, vol. 32. American Mathematical Society.
\emph{International Press}, Providence (2002).

\bibitem{l} L. H$\rm{\ddot{o}}$mander, Lectures on Nonlinear Hyperbolic Differential Equations, Math\'{e}matiques
and Applications 26. \emph{Springer-verlag}, Berlin (1997).

\bibitem{m} M. Tsuji, Formation of singularities for Monge-Amp$\grave{\rm{e}}$re equations, \emph{Bull. Sci.
Math.} {\bf 119} (1995) 433--457.

\bibitem{pw} M. H. Protter, H. F. Weinberger, Maximum Principles in
Differential Equations, \emph{Springer-verlag}, New York (1984).

\bibitem{rm} M. Gage, R. Hamilton, The heat equation shring convex plane curves, \emph{J. Differ. Geom.} {\bf 23} (1986) 69--96.

\bibitem{rs} R. Schneider, Convex bodies: The Brum-Minkowski Theory, \emph{Cambridge University Press}, (1993).

\bibitem{dd} D. DeTurck, Some regularity theorems in Riemannian geometry, \emph{Ann. Sci. \'{E}cole Norm. Sup.}{\bf 14} (1981) 249--260.


\end{thebibliography}
 \end{document}